\documentclass[11pt,a4paper,twoside]{article}
\usepackage{graphics}
\usepackage{amssymb}
\usepackage{amsmath}
\usepackage{mathrsfs}
\usepackage{makeidx}
\usepackage{color}
\usepackage{graphicx}
\usepackage{subfigure}
\usepackage{multicol}
\usepackage{multirow}
\usepackage{float}
\usepackage{booktabs}
\usepackage[colorlinks,
            linkcolor=red,
            anchorcolor=blue,
            citecolor=blue
            ]{hyperref}
\usepackage{caption}
\usepackage{algorithm}
\usepackage{algorithmic}
\usepackage{latexsym, cite, bm, amsthm, amsmath}
\usepackage{amsthm}
\usepackage{marginnote}
\usepackage{mathrsfs}

\def \[{\begin{equation}}
\def \]{\end{equation}}
\newtheorem{thm}{Theorem}[section]
\newtheorem{cor}{Corollary}[section]
\newtheorem{lem}{Lemma}[section]

\newtheorem{alg}{Algorithm}[section]
\newtheorem{rem}{Remark}[section]
\newtheorem{definition}{Definition}[section]
\newtheorem{exam}{Example}[section]
\numberwithin{equation}{section}

\title{An inexact framework of the Newton-based matrix splitting iterative method for the generalized absolute value equation}

\usepackage[marginal]{footmisc}
\usepackage{authblk}
\author[a,b]{Dongmei Yu\thanks{Supported partially by China Postdoctoral Science Foundation (2019M650449), the Natural Science Foundation of Liaoning Province (Nos. 2020-MS-301, 2022lslwtkt-069) and the Ministry of Education in China of Humanities and Social Science Project (No. 21YJCZH204). Email: {yudongmei1113@163.com}.}}
\author[c]{Cairong Chen\thanks{Corresponding author. Supported partially by the National Natural Science Foundation of China(No. 11901024) and the Natural Science Foundation of Fujian Province (No. 2021J01661). Email: {cairongchen@fjnu.edu.cn}.}}
\author[b]{Deren Han\thanks{Supported partially by the National Natural Science Foundation of China (Grant Nos. 12131004 and 11625105). Email: {handr@buaa.edu.cn}.}}
\affil[a]{Institute for Optimization and Decision Analytics, Liaoning Technical University, Fuxin, 123000, P.R. China.}
\affil[b]{LMIB of the Ministry of Education, School of Mathematical Sciences, Beihang University, Beijing 100191, P.R. China. }
\affil[c]{College of Mathematics and Informatics, FJKLMAA and Center for Applied Mathematics of Fujian Province, Fujian Normal University, Fuzhou 350007, P.R. China.}

\begin{document}
\date{\today}
\maketitle


\begin{quote}
{\bf Abstract:} An inexact framework of the Newton-based matrix splitting (INMS) iterative method is developed to solve the generalized absolute value equation, whose exact version was proposed by Zhou, Wu and Li [H.-Y. Zhou, S.-L. Wu and C.-X. Li, \textit{J. Comput. Appl. Math.}, 394 (2021), 113578]. Global linear convergence of the INMS iterative method is investigated in detail. Some numerical results are given to show the superiority of the INMS iterative method.

{\small

}

\end{quote}
\section{Introduction}\label{sec:intro}
In this paper, we concentrate on the solution of the generalized absolute value equation (GAVE)
\begin{equation}\label{eq:gave}
Ax-B|x|-b=0,
\end{equation}
where $A, B\in \mathbb{R}^{n\times n}$, $b\in \mathbb{R}^n$ are known and $|x|=(|x_1|, |x_2|,\cdots, |x_n|)^\top$ denotes
the component-wise absolute value of the unknown vector $x=(x_1, x_2,\cdots, x_n)^\top \in \mathbb{R}^{n}$.
The GAVE~\eqref{eq:gave} is first introduced in~\cite{rohn2004} and further investigated in \cite{mang2007,hlad2018,prok2009, wush2020} and references therein. In the special case $B=I$ or $B$ is nonsingular, the GAVE~\eqref{eq:gave} can be reduced to the absolute value equation (AVE)
\begin{equation}\label{eq:ave}
Ax-|x|-b=0.
\end{equation}
The main significance of the GAVE~\eqref{eq:gave} and the AVE~\eqref{eq:ave} arises from the fact that they have many applications in optimization fields such as linear programming problem, bimatrix games, mixed integer programming, complementarity problem, quadratic programming and others, see, e.g., \cite{mang2006,mang2007,abhm2018,manga2009,rohn1989,prok2009} and references therein. Particularly, if $B=0$, then the GAVE~\eqref{eq:gave} reduces to the linear system $Ax=b$, which plays a significant role in scientific computation.

Over the last twenty years, the GAVE~\eqref{eq:gave} and the AVE~\eqref{eq:ave} have been extensively investigated, and the research work mainly focuses on two aspects: to provide theoretical analysis and to explore efficient numerical methods. On the theoretical side, many studies have focused on the equivalent reformulations of the GAVE~\eqref{eq:gave} or the AVE~\eqref{eq:ave}, and detected the existence and nonexistence of solutions, see, e.g., \cite{mang2006,mang2007,prok2009,wu2018,wu2019,wush2020,mezzadri2020} and references therein.  Especially, it has been proved in~\cite{mang2007} that determining the existence of a solution of the general GAVE \eqref{eq:gave} or AVE \eqref{eq:ave} is NP-hard. Furthermore, in~\cite{prok2009}, it has been shown that checking whether the GAVE \eqref{eq:gave} or the AVE \eqref{eq:ave} has a unique solution or multiple solutions is NP-complete. On the numerical side, it focuses on exploring efficient numerical algorithms for solving the GAVE \eqref{eq:gave} and the AVE \eqref{eq:ave}. For example, the Newton-type methods \cite{mang2009,caqz2011,crfp2016,wangcao2019,zainali2018}, the neural network approaches \cite{maee2017,maer2018,chen2021}, the SOR-like iterations \cite{kema2017,dongshao2020,guwl2019}, the concave minimization methods \cite{mang2007a,mang2007,zamani2021,abhm2018}, the conjugate gradient method \cite{raam2019} and others, see, e.g., \cite{guhl2017,ke2020,yuch2020,sayc2018,edhs2017,iqia2015} and references therein. In the following, we go in for a closer look on some Newton-type methods.

By considering the GAVE~\eqref{eq:gave} as a system of nonlinear equations
\begin{equation}\label{nonli}
F(x)=0 \quad \text{with} \quad F(x):= Ax-B|x|-b,
\end{equation}
some Newton-type algorithms for solving nonsmooth equations are developed to find a solution of the GAVE~\eqref{eq:gave} and the AVE~\eqref{eq:ave}. Mangasarian~\cite{mang2009} utilized the generalized Jacobian $\partial|x|$ of $|x|$ based on a subgradient of its components and directly proposed the generalized Newton (GN) iterative method to solve the AVE \eqref{eq:ave}. Hu et al.~\cite{huhuzh2012} then extended the GN iteration to solve the GAVE~\eqref{eq:gave}. As the Jacobian matrix of the GN iteration is changed at each iterative step, it is undesirable for solving large-scaled problems, especially if the Jacobian matrix is ill-conditioned. To overcome this shortcoming, modified Newton-type (MN) iterative methods are developed for solving the GAVE~\eqref{eq:gave} \cite{wangcao2019}. Then a new MN (NMN) iterative method is proposed in \cite{li2021}, which is more balanced than the MN method. Subsequently, a more general Newton-based matrix splitting (NMS) method is established to solve the GAVE~\eqref{eq:gave}~\cite{wu2021n}. As shown in \cite{wu2021n}, by choosing suitable matrix splitting, the NMS method can include the Picard method \cite{rohn2014} and the MN method, and it also generates some relaxation versions. Furthermore, we find that the NMS method can comprise the NMN method and a special case of the Douglas-Rachford splitting method \cite{chenyu2021}.

However, each iteration of the NMS method requires to solve exactly a linear system, which may be intractable for large-scaled problems. This motivates us to develop an inexact framework of the NMS (INMS) iterative method for solving the GAVE~\eqref{eq:gave}, which can provide inexact versions of the afore-mentioned exact methods. Each step of the INMS method only requires to inexactly solve the involved linear system and a relative residual error tolerance is  adopted.


The remainder of this paper is organized as follows. Section~\ref{sec:imn} is devoted to the development of the INMS iterative method for solving the GAVE~\eqref{eq:gave}. In Section~\ref{sec:cov}, the global linear convergence of the INMS iterative method is explored. Section~\ref{sec:numer} reports the numerical results. Finally, some concluding remarks are given in Section~\ref{sec:conclusion}.

\textbf{Notations}. Throughout this paper, we adopt the following notations. Let $\mathbb{R}^{n\times n}$ be the set of all $n \times n$ real matrices and $\mathbb{R}^{n}= \mathbb{R}^{n\times 1}$. The identity matrix with suitable dimension is denoted by $I$. The transposition of a matrix or a vector is denoted by $\cdot ^\top$. For a vector $x=(x_1,x_2,\cdots,x_n)^\top \in \mathbb{R}^{n}$, $x_i$ refers to its $i$-th entry, $|x|$ is in $\mathbb{R}^{n}$ with its $i$-th entry $| x_i |$, and  $| \cdot |$ denotes the absolute value for real scalar. For $x \in \mathbb{R}^{n}$, $\|x\|$ denotes its $2$-norm and $\textbf{diag}(x)$ represents a diagonal matrix with $x_i$ as its diagonal entries for every $i=1,2,\cdots,n$. For $M \in \mathbb{R}^{n\times n}$, $\|M\|$ denotes the spectral norm of $M$ and is defined by $\| M \|\doteq \max \left\{ \| M x \| : x \in \mathbb{R}^{n}, \|x\|=1 \right\}$.

\section{The INMS iterative method}\label{sec:imn}
In this section, an inexact framework of the NMS iterative method for solving the GAVE~\eqref{eq:gave} is developed. To this end, we start with a short review of the exact NMS iterative method proposed in \cite{wu2021n}.

For a system of nonlinear equations with the nonlinear function which is divided into two parts, one is differentiable and another is non-differential but Lipschitz continuous, an idea to construct numerical methods is to separate the differential part and the non-differential part \cite{han2001}. Recall~\eqref{nonli}, let $A = M - N$ and
\begin{equation*}
F(x) = P(x) + Q(x)~\text{with}~P(x)=  (\Omega+M) x~\text{and}~Q(x)= -(\Omega+N) x - B|x|-b,
\end{equation*}
where $\Omega \in \mathbb{R}^{n\times n}$ is a given matrix such that the Jacobian matrix $P'(x) = \Omega + M$ is nonsingular. Based on the afore-mentioned idea, the NMS iterative method is established and the detail is given in Algorithm~\ref{alg:nms}.

\begin{alg}[\!\!\cite{wu2021n}]\label{alg:nms}{\bf $($The NMS iterative method$)$} Let $A=M-N$ be a splitting of the matrix $A \in \mathbb{R}^{n\times n}$, $B \in \mathbb{R}^{n\times n}$ and $b\in \mathbb{R}^{n}$. Given a matrix $\Omega \in \mathbb{R}^{n \times n}$ such that $\Omega+ M$ is invertible. Assume that $x^0\in \mathbb{R}^n$ is an arbitrary initial guess.  For $k = 0,\,1,\,2,\,\cdots$ until the iterative sequence $\{x^k\}$ is convergent, compute	
\begin{equation}\label{eq:innms-iter}
	x^{k+1}={(\Omega+M)}^{-1}[(\Omega +N) x^k +B|x^k|+b].
	\end{equation}
\end{alg}

The iterative sequence $\{x^k\}$ generated by Algorithm~\ref{alg:nms} has the following general convergence property.

\begin{thm}[\!\!\cite{wu2021n}]\label{thm:iterpro} Let $A, B \in \mathbb{R}^{n\times n}$ and $\Omega \in \mathbb{R}^{n\times n}$ is a given matrix. Assume that $A = M-N$ and $\Omega+M$ is nonsingular.
If
\begin{equation}\label{eq:exact}
\|(M+\Omega)^{-1}\|(\|N+\Omega\|+\|B\|)<1,
\end{equation}
then the sequence $\{x^k\}$ generated by Algorithm~\ref{alg:nms} converges linearly from any starting point to a solution $x^*$ of the GAVE~\eqref{eq:gave}.
\end{thm}

From Algorithm~\ref{alg:nms}, at each iterative step, the NMS iterative method requires the exact solution of the linear system with coefficient matrix $\Omega +M$, which might be computationally expensive or impractical in actual implementations. To alleviate the burden of each step and to
further improve the computational efficiency, a new algorithm adopting approximate solution of the above linear system is much desirable. To this end, in Algorithm~\ref{alg:inms}, we develop the INMS iterative method adapted to solving the GAVE~\eqref{eq:gave}.

\begin{alg}\label{alg:inms}{\bf $($The INMS iterative method$)$} Let $A=M-N$ be a splitting of $A\in\mathbb{R}^{n\times n}$, $B \in \mathbb{R}^{n\times n}$ and $b\in \mathbb{R}^{n}$. Given an initial guess $x^0\in \mathbb{R}^n$, a residual relative error tolerance $\theta \in [0,1)$ and a matrix $\Omega \in \mathbb{R}^{n \times n}$. Assume that $\Omega+ M$ is invertible, for $k = 0,\,1,\,2,\,\cdots$ until the iterative sequence $\{x^k\}$ is convergent, compute $x^{k+1}$ such that
	\begin{equation}\label{eq:inms-iter}
	\|(\Omega+M)x^{k+1}-[(\Omega+N) x^k +B|x^k|+b]\| \leq \theta \| F(x^k)\| .
	\end{equation}
\end{alg}

It is noteworthy that in the case $\theta = 0$, the INMS iterative method will retrieve the NMS iterative method. Thus, Algorithm~\ref{alg:inms} also encompasses the inexact versions of the methods mentioned in the following Remarks~\ref{rem:spec}--\ref{rem:spec3} (which appeared in \cite{wu2021n}, except Remarks~\ref{rem:spec1}--\ref{rem:spec3}).

\begin{rem}\label{rem:spec}
Let $A = D - L - U$ with $D = diag(A)$, $L$ and $U$ represent the strictly lower-triangular and upper-triangular part of $-A$, respectively. It has been mentioned in \cite{wu2021n} that Algorithm~\ref{alg:nms} will reduce to the following methods.
\begin{itemize}
  \item[(a)] If $M = A$ and $\Omega= N = 0$, then Algorithm~\ref{alg:nms} becomes the Picard method~\cite{rohn2014}:
  \begin{equation}\label{eq:picard}
	x^{k+1}= A^{-1}(B|x^k|+b).
  \end{equation}

  \item[(b)] If $M = A$ and $N = 0$, then Algorithm~\ref{alg:nms} turns into the MN method~\cite{wangcao2019}:
  \begin{equation}\label{eq:mn}
  x^{k+1} = (\Omega +A)^{-1}(\Omega x^k + B|x^k|+b).
 \end{equation}

  \item[(c)] If $M = D$ and $N = L+U$, then Algorithm~\ref{alg:nms} reduces to the Newton-based Jacobi (NJ) method:
  \begin{equation}\label{eq:nj}
  x^{k+1} = (\Omega +D)^{-1}[(\Omega+L+U) x^k + B|x^k|+b)].
  \end{equation}

  \item[(d)] If $M = D-L$ and $N = U$, then Algorithm~\ref{alg:nms} becomes the Newton-based Gauss-Seidel (NGS) method:
  \begin{equation}\label{eq:ngs}
  x^{k+1} = (\Omega +D-L)^{-1}[(\Omega+U) x^k + B|x^k|+b)].
  \end{equation}

  \item[(e)] If $M = \frac{1}{\alpha}D - L$ and $N = (\frac{1}{\alpha} - 1)D + U$, then Algorithm~\ref{alg:nms} reduces to the Newton-based SOR (NSOR) method:
      $$
      x^{k+1} = (D + \alpha\Omega -\alpha L)^{-1}[(\alpha\Omega+\alpha U+ (1-\alpha)D) x^k +\alpha( B|x^k|+b)].
      $$

      \item[(f)] If $M = \frac{1}{\alpha}(D - \beta L)$ and $N = \frac{1}{\alpha}\left[(1-\alpha)D + (\alpha - \beta)L + \alpha U\right]$, then Algorithm~\ref{alg:nms} turns into the Newton-based AOR (NAOR) method:
      $$
      x^{k+1} = (D + \alpha\Omega -\beta L)^{-1}[(\alpha\Omega+\alpha U+ (1-\alpha)D + (\alpha-\beta)L) x^k +\alpha( B|x^k|+b)].
      $$

      \item[(g)] Let $H = \frac{1}{2}(A + A^\top)$ and $S = \frac{1}{2}(A - A^\top)$. If $M = H$ and $N = -S$, then Algorithm~\ref{alg:nms} reduces to the Newton-based Hermitian and skew-Hermitian  method:
      $$
      x^{k+1} = (\Omega +H)^{-1}[(\Omega-S) x^k +B|x^k|+b].
      $$
\end{itemize}
\end{rem}
\begin{rem}\label{rem:spec1}
If $M = \frac{1}{2}(A-\Omega)$ and $N=-\frac{1}{2}(A+\Omega)$, then Algorithm~\ref{alg:nms} turns into the NMN method~\cite{li2021}:
\begin{equation}\label{eq:nmn}
  x^{k+1}={(\Omega+A)}^{-1}[(\Omega-A) x^k +2(B|x^k|+b)].
\end{equation}
\end{rem}
\begin{rem}\label{rem:spec3}
Let $M =A$, $N=0$ and $\Omega=(\frac{2}{\gamma}-1)A$ with $\gamma \in (0,2)$. If $A\in \mathbb{R}^{n\times n}$ is nonsingular and $B=I$, then Algorithm~\ref{alg:nms} becomes the Douglas-Rachford splitting method for solving the AVE~\eqref{eq:ave}~\cite{chenyu2021}:
\begin{equation}\label{eq:drs}
  x^{k+1}=(1-\frac{1}{2}\gamma)x^k+\frac{1}{2}\gamma A^{-1}(|x^k|+b).
\end{equation}
\end{rem}

In the next section, we will devote ourselves to the study of the convergence of the INMS method.
\section{Convergence analysis}\label{sec:cov}
In this section, we will analyze the general convergence of the INMS iterative method in the context of solving the GAVE~\eqref{eq:gave}. Throughout this paper, we assume that the solution set of the GAVE~\eqref{eq:gave} is nonempty.

Before establishing the convergence of the sequence $\{x^k\}$ generated by Algorithm~\ref{alg:inms}, a family of mappings are defined according to \eqref{eq:inms-iter} and  their properties are studied.

\begin{definition}\label{eq:def}
For $\theta \in [0,1)$, $\mathcal{N}_\theta$ is the family of mappings $N_\theta:$  $\mathbb{R}^{n} \rightarrow \mathbb{R}^{n}$ such that
	\begin{equation*}
	\|(\Omega+M)N_\theta(x)-[(\Omega+N) x +B|x|+b]\| \leq \theta \| F(x)\| ,\quad \forall x \in \mathbb{R}^{n}.
	\end{equation*}
\end{definition}

If $\Omega +M$ is invertible, then the family $\mathcal{N}_0$ only has a single element, that is, the exact NMS iterative map $N_0: \mathbb{R}^{n} \rightarrow \mathbb{R}^{n}$ defined by
\begin{equation*}
	N_0(x)=(\Omega+M)^{-1}[(\Omega+N) x +B|x|+b].
\end{equation*}
In light of Definition~\ref{eq:def}, $\mathcal{N}_0 \subseteq  \mathcal{N}_\theta \subseteq  \mathcal{N}_{\theta'}$ for $0 \leq \theta \leq  \theta' <1$. Hence $\mathcal{N}_\theta$ is nonempty for all $\theta \in [0,1)$. More specifically, for $x\in \mathbb{R}^n$, any $\theta\in [0,1)$ and $N_\theta \in \mathcal{N}_\theta$, we have $N_\theta(x) = x$ if and only if $F(x) = 0$.

According to \eqref{eq:inms-iter} and Definition~\ref{eq:def}, the outcome of the INMS iterative method is
\begin{equation}\label{eq:14}
x^{k+1}=N_\theta(x^k),\quad k=0,1,2,\cdots
\end{equation}
with some $N_\theta \in \mathcal{N}_\theta$ and $\theta \in [0,1)$. Thus, the following lemma lays the foundation of our convergence analysis hereinafter.
\begin{lem}\label{imncon}
 Assume that $\Omega+M$ is invertible. Let $\theta \in [0,1)$ and $N_\theta \in \mathcal{N}_\theta$ defined as in Definition~\ref{eq:def}. If $F(x^*)=0$, then for every $x \in \mathbb{R}^{n}$  we obtain
\begin{equation}\label{eq:nx}
\|N_\theta(x)-x^*\| \leq \|(\Omega+ M)^{-1}\|  \left[\theta (\|\Omega+M\|+\|\Omega+N\|+\|B\|) + \|B\|+\|\Omega+N \| \right] \|x-x^*\|.
\end{equation}
\end{lem}

\begin{proof}
Let $x \in \mathbb{R}^{n}$. Taking into account that $F(x^*)=0$, it is easy to see that
\begin{align*}
N_\theta(x)-x^* &=(\Omega +M)^{-1}\{(\Omega+M)N_\theta(x)-(\Omega+N) x-B|x|-b \\
&\quad+[F(x^*)-F(x)-(\Omega+M)(x^*-x)]\},
\end{align*}
from which we immediately have
\begin{align}\nonumber
\|N_\theta(x)-x^*\| & \leq \|(\Omega+M)^{-1}\| (\|(\Omega+M)N_\theta(x)-(\Omega+N) x-B|x|-b\| \\ \label{eq:ineql}
&\quad+\|F(x^*)-F(x)-(\Omega+M)(x^*-x)\|).
\end{align}
The combination of Definition~\ref{eq:def} and the inequality \eqref{eq:ineql} leads to
\begin{equation}\label{eqrif}
\|N_\theta(x)-x^*\|  \leq \|(\Omega+M)^{-1}\| (\theta \|F(x)\|  +\|F(x^*)-F(x)-(\Omega+M)(x^*-x)\|).
\end{equation}
On the other hand, it follows from $F(x^*)=0$ that
$$ F(x)=(\Omega+M)(x-x^*)-[F(x^*)-F(x)-(\Omega+M)(x^*-x)], $$
from which we get
\begin{equation}\label{eq:10}
\|F(x)\| \leq \|\Omega+M\| \|x-x^*\|+ \|F(x^*)-F(x)-(\Omega+M)(x^*-x)\|.
\end{equation}
Furthermore, some manipulation yields
\begin{align}\nonumber
F(x^*)&-F(x)-(\Omega+M)(x^*-x) \\\nonumber
&=(\Omega+M)x^*-(\Omega+N)x^*-B|x^*|-b-[(\Omega+M)x-(\Omega+N)x\\\nonumber
&-B|x|-b]-(\Omega+M)(x^*-x)\\\nonumber
&=(\Omega+N)(x -x^*)+B(|x|-|x^*|),
\end{align}
from which we can conclude that
\begin{equation}\label{eq:11}
\|F(x^*)-F(x)-(\Omega+M)(x^*-x)\| \leq (\|\Omega+N\|+\|B\|) \|x-x^*\|,
\end{equation}
where $\||x|-|x^*|\|\le \|x-x^*\|$ is utilized.

Inequalities \eqref{eq:10} and \eqref{eq:11} collectively imply that
\begin{equation}\label{eq:12}
\|F(x)\| \leq (\|\Omega+M\|+\|\Omega+N\|+\|B\|) \|x-x^*\|.
\end{equation}
Substituting the inequalities \eqref{eq:11} and \eqref{eq:12} into inequality \eqref{eqrif}, we conclude that
\begin{align*}
\|N_\theta(x)-x^*\| & \leq \|(\Omega+M)^{-1}\| [\theta (\|\Omega+M\|+\|\Omega+N\|+\|B\|)\|x-x^*\|+(\|\Omega+N\|+\|B\|)\|x-x^*\|] \\
&=\|(\Omega+M)^{-1}\|\left[\theta (\|\Omega+M\|+\|\Omega+N\|+\|B\|)+\|\Omega+N\|+\|B\|\right]\|x-x^*\|,
\end{align*}
which completes the proof.
\end{proof}

Now, we are in position to prove the main results of this section.

\begin{thm}\label{thm:newimn}
Let $A=M-N$ be a splitting of $A\in \mathbb{R}^{n\times n}$, $B\in \mathbb{R}^{n\times n}$, $b\in \mathbb{R}^{n}$, $\theta \in [0,1)$. Assume that $\Omega \in \mathbb{R}^{n\times n}$ is a given matrix such that $\Omega+M$ is invertible. Then, the sequence $\{x^k\}$ generated by Algorithm~\ref{alg:inms} with any starting point $x^0\in \mathbb{R}^{n}$ is well defined and it has
\begin{equation}\label{eq:newx}
\|x^{k+1}-x^*\| \leq \|(\Omega+M)^{-1}\|  \left[\theta (\|\Omega+M\|+\|\Omega+N\|+\|B\|)+\|\Omega+N\|+\|B\|\right]\|x^k-x^*\|
\end{equation}
for $k=1,2,\cdots$.
Furthermore, if
\begin{equation}\label{eq:15}
\|(\Omega+M)^{-1}\| < \frac{1}{\theta (\|\Omega+M\|+\|\Omega+N\|+\|B\|)+\|\Omega+N\|+\|B\|},
\end{equation}
then the sequence $\{x^k\}$ converges linearly to $x^* \in \mathbb{R}^{n}$, a solution of the GAVE \eqref{eq:gave}.
\end{thm}

\begin{proof}
For any starting point $x^0\in \mathbb{R}^{n}$, by Definition~\ref{eq:def} and \eqref{eq:inms-iter}, the well-definedness of $\{x^k\}$ follows from invertibility of $\Omega+M$. Since $x^*$ is the solution of \eqref{eq:gave}, together with $F(x^*)=0$, we then conclude that for $k=1,2,\cdots$, the sequence $\{x^k\}$ satisfies \eqref{eq:newx} according to \eqref{eq:14} and \eqref{eq:nx}. On the other hand, taking~\eqref{eq:15} into account, we immediately obtain,
$$ \|(\Omega+M)^{-1}\| [\theta (\|\Omega+M\|+\|\Omega+N\|+\|B\|)+\|\Omega+N\|+\|B\|]<1,$$
which combining with \eqref{eq:newx} means that the sequence $\{x^k\}$ converges linearly to $x^*$.
\end{proof}

\begin{rem}
If $\theta = 0$, then \eqref{eq:15} reduces to \eqref{eq:exact}, a sufficient convergence condition for the NMS iterative method proposed in  \cite{wu2021n}. However, our proof here seems differ from that of Theorem~3.1 in \cite{wu2021n}.
\end{rem}

In the light of the Banach perturbation~\cite {gova2009}, Theorem~\ref{thm:newimna} can be given as well.
\begin{thm}\label{thm:newimna}
Let $A \in \mathbb{R}^{n\times n}$ be nonsingular and $A=M-N$, where $M$ is invertible. Assume that~ $\Omega \in \mathbb{R}^{n\times n}$ is a given matrix such that $\Omega + M$ is invertible. If
	\begin{equation}\label{eq:ncont}
	\|M^{-1}\| < \frac{1}{\theta (\|\Omega+M\|+\|\Omega+N\|+\|B\|)+\|\Omega+N\|+\|B\|+\|\Omega\|},
	\end{equation}
	then the INMS iterative method converges linearly from any starting point to a solution $x^*$ of the GAVE~\eqref{eq:gave}.
\end{thm}

\begin{proof}
Following from the proof of Theorem~\ref{thm:newimn}, we can conclude from~\eqref{eq:newx} that if
$$\|(\Omega+M)^{-1}\|  \left[\theta (\|\Omega+M\|+\|\Omega+N\|+\|B\|)+\|\Omega+N\|+\|B\|\right]<1, $$
then the INMS iterative method is linearly convergent. According to \eqref{eq:ncont} and the Banach perturbation~\cite[Lemma~2.3.3]{gova2009}, it follows that
\begin{align}
\|(M+\Omega)^{-1}\| &\leq \frac{\|M^{-1}\|}{1-\|M^{-1}\| \cdot \|\Omega\|} \\\nonumber
&< \frac{\frac{1} {\theta (\|\Omega+M\|+\|\Omega+N\|+\|B\|)+\|\Omega+N\|+\|B\|+\|\Omega\|}}
 {1-\frac{\|\Omega\|}{\theta (\|\Omega+M\|+\|\Omega+N\|+\|B\|)+\|\Omega+N\|+\|B\|+\|\Omega\|}} \\\nonumber
&=\frac{1}{\theta (\|\Omega+M\|+\|\Omega+N\|+\|B\|)+\|\Omega+N\|+\|B\|}.
\end{align}
Therefore, the INMS iterative method converges linearly from any starting point to a solution $x^*$ of the GAVE~\eqref{eq:gave} provided that the condition~\eqref{eq:ncont} is satisfied.
\end{proof}

According to Remarks~\ref{rem:spec}--\ref{rem:spec1}, the following corollaries can be derived.

\begin{cor}\label{cor:cor1}
Let $A,B\in \mathbb{R}^{n\times n}$, $b\in \mathbb{R}^{n}$, $\theta \in [0,1) $. Assume that $\Omega \in \mathbb{R}^{n\times n}$ is a given matrix such that $\Omega+A$ is invertible. If	
	\begin{equation*}
	\|(\Omega+ A)^{-1}\| < \frac{1}{\|B\|+\|\Omega\|+ \theta(\|\Omega+A\|+\|B\|+\|\Omega\|)},
	\end{equation*}
then the inexact MN iterative method converges linearly from any starting point to a solution $x^*$ of the GAVE~\eqref{eq:gave}.
\end{cor}
Particularly, the Theorem~3.1 proposed in \cite{wangcao2019} can be derived from Corollary~\ref{cor:cor1} whenever $\theta = 0$.

\begin{cor}\label{cor:cor2}
Assume that $A\in \mathbb{R}^{n\times n}$ is invertible, $B\in \mathbb{R}^{n\times n}$, $b\in \mathbb{R}^{n}$, $\theta \in [0,1)$. $\Omega \in \mathbb{R}^{n\times n}$ is a given matrix such that $A + \Omega$ is nonsingular. If
	\begin{equation*}
	\|A^{-1}\| < \frac{1}{\|B\|+2\|\Omega\|+ \theta(\|\Omega+A\|+\|B\|+\|\Omega\|)},
	\end{equation*}
	then the inexact MN iterative method converges linearly from any starting point to a solution $x^*$ of the GAVE~\eqref{eq:gave}.
\end{cor}
Specially, if $\theta = 0$, then the Theorem~3.2 proposed in \cite{wangcao2019} can be derived from Corollary~\ref{cor:cor2}.

\begin{cor}\label{cor:cor3}
Let $A,B\in \mathbb{R}^{n\times n}$, $b\in \mathbb{R}^{n}$, $\theta \in [0,1) $. Assume that  $\Omega \in \mathbb{R}^{n\times n}$ is a given matrix such that $\Omega+A$ is invertible. If
\begin{equation*}
	\|(\Omega +A)^{-1}\| < \frac{1}{2\|B\|+\|\Omega-A\|+ \theta(\|\Omega+A\|+2\|B\|+\|\Omega-A\|)},
	\end{equation*}
or if $A \in \mathbb{R}^{n\times n}$ be invertible and
	\begin{equation*}
	\|A^{-1}\| < \frac{1}{2\|B\|+\|\Omega\|+ \|\Omega-A\|+\theta(\|\Omega+A\|+2\|B\|+\|\Omega-A\|)},
	\end{equation*}
    then the inexact NMN iterative method converges linearly from any starting point to a solution $x^*$ of the GAVE~\eqref{eq:gave}.
\end{cor}
If $\theta = 0$, then the Theorems~1 and 2 proposed in \cite{li2021} can be derived from Corollary~\ref{cor:cor3}.

\begin{cor}\label{cor:cor4}
Assume that $A \in \mathbb{R}^{n\times n}$ is invertible, $B \in \mathbb{R}^{n\times n}$, $b\in \mathbb{R}^{n}$ and $\theta \in [0,1)$. If
	\begin{equation*}
	\|A^{-1}\| < \frac{1}{\|B\|+ \theta(\|A\|+\|B\|)},
	\end{equation*}
then the inexact Picard iterative method converges linearly from any starting point to a solution $x^*$ of the GAVE~\eqref{eq:gave}.
\end{cor}

Since the GAVE~\eqref{eq:gave} reduces to the AVE~\eqref{eq:ave} by simply letting $B = I$, the INMS iterative method can be directly used to solve the AVE~\eqref{eq:ave} and the following corollary can be obtained.

\begin{cor}\label{cor:cor5}
Let $A=M - N\in \mathbb{R}^{n\times n}$, $b\in \mathbb{R}^{n}$, $\theta \in [0,1) $. Assume that $\Omega \in \mathbb{R}^{n\times n}$ is a given matrix such that $\Omega+ M$ is invertible. Then if
	\begin{equation*}
	\|(\Omega + M)^{-1}\| < \frac{1}{\theta(\|\Omega+ M\|+\|\Omega + N\| + 1)+\|\Omega + N\|+1},
	\end{equation*}
	or if $M$ is invertible and
	\begin{equation*}
	\|M^{-1}\| < \frac{1}{\theta(\|\Omega+ M\|+\|\Omega + N\| +1)+\|\Omega+N\|+\|\Omega\|+1},
	\end{equation*}
	then the INMN iterative method converges linearly from any starting point to a solution $x^*$ of the AVE~\eqref{eq:ave}.
\end{cor}

Meanwhile, some corresponding results proposed in \cite{wangcao2019,wu2021n,li2021} can be derived from Corollary~\ref{cor:cor5} and the detail is omitted here. According to Remark~\ref{rem:spec3} and Corollary~\ref{cor:cor5}, the following result can be obtained.
\begin{cor}\label{cor:cor6}
Let $\gamma\in (0,2)$, $\theta\in [0,1)$ and $A$ be nonsingular. If
\begin{equation*}
	\|A^{-1}\| < \frac{1}{\theta[(2-\frac{\gamma}{2})\|A\| +1]+2(1-\frac{\gamma}{2})\|A\|+1},
	\end{equation*}
or if
\begin{equation*}
	\|A^{-1}\| < \frac{1}{\theta[(\frac{4}{\gamma}-1)\|A\| +1]+2(\frac{2}{\gamma}-1)\|A\|+1},
	\end{equation*}
then the inexact Douglas-Rachford splitting method converges linearly from any starting point to a solution $x^*$ of the AVE~\eqref{eq:ave}.
\end{cor}
It should mention that the condition proposed in Corollary~\ref{cor:cor6} is different from that of \cite{chenyu2021}.

Finally, we will discuss the convergence conditions of the INMS iterative method~\eqref{eq:inms-iter} for solving the GAVE~\eqref{eq:gave} and the AVE~\eqref{eq:ave} when the matrix $\Omega=\omega I(\omega > 0)$ is a positive scalar matrix.

\begin{thm}\label{thm:scal}
 Let $A \in \mathbb{R}^{n\times n}$ be a positive definite matrix and $A = H+S$ be its a splitting with $H=\frac{1}{2}(A+A^\top)$ and $S=\frac{1}{2}(A-A^\top)$. Let $\lambda_{\min}$ and $\lambda_{\max}$ be the minimum eigenvalue and the maximum eigenvalue of the matrix $H$, respectively. Assume that $\mu_{\max}$ is the maximum value of the absolute values of the eigenvalues of the matrix $S$, $\|B\|=\tau$ and $\Omega=\omega I (\omega> 0)\in \mathbb{R}^{n\times n}$. If
\begin{align*}
\omega+\lambda_{\min}-\tau >\sqrt{\omega^2+\mu_{\max}^2} +\theta(\omega+\lambda_{\max}+\tau+\sqrt{\omega^2+\mu_{\max}^2}),
\end{align*}
then the INMS iterative method with $M = H$ and $N=-S$ converges linearly from any starting point to a solution $x^*$ of the GAVE \eqref{eq:gave}.
\end{thm}

\begin{proof}
Since $M=H$, $N=-S$ and $\Omega=\omega I $, one can derive that
\begin{align*}
&\|(\Omega+M)^{-1}\| [\theta (\|\Omega+M\|+\|\Omega+N\|+\|B\|)+\|\Omega+N\|+\|B\|]\\
=&\|(\omega I+H)^{-1}\| [\theta (\|\omega I+H\|+\|\omega I-S\|+\|B\|)+\|\omega I-S\|+\|B\|]\\
=& \frac{\theta(\omega+\lambda_{\max}+\sqrt{\omega^2+\mu_{\max}^2}+\tau)+\sqrt{\omega^2+\mu_{\max}^2}+\tau}{\omega+\lambda_{\min}},
\end{align*}
from which and \eqref{eq:15} we obtain that if
$$\omega+\lambda_{\min}-\tau >\sqrt{\omega^2+\mu_{\max}^2} +\theta(\omega+\lambda_{\max}+\tau+\sqrt{\omega^2+\mu_{\max}^2}),$$
then the INMS iterative method converges linearly from any starting point to a solution $x^*$ of the GAVE~\eqref{eq:gave}.
\end{proof}
In particular, if $B=I$, the following corollary can be obtained.
\begin{cor}\label{cor4}
Let $A \in \mathbb{R}^{n\times n}$ be a positive definite matrix and $A = H+S$ be its a splitting with $H=\frac{1}{2}(A+A^T)$ and $S=\frac{1}{2}(A-A^T)$. Let $\lambda_{\min}$ and $\lambda_{\max}$ be the minimum eigenvalue and the maximum eigenvalue of the matrix $H$, respectively. Assume that $\mu_{\max}$ is the maximum value of the absolute values of the eigenvalues of the matrix $S$ and $\Omega=\omega I (\omega> 0)\in \mathbb{R}^{n\times n}$. If
\begin{align*}
\omega+\lambda_{\min}-1>\sqrt{\omega^2+\mu_{\max}^2} +\theta(\omega+\lambda_{\max}+1+\sqrt{\omega^2+\mu_{\max}^2}),
\end{align*}
then the INMS iterative method converges linearly from any starting point to the unique solution $x^*$ of the AVE~\eqref{eq:ave}.
\end{cor}

\section{Numerical results}\label{sec:numer}
In this section, some numerical results will be presented to illustrate the efficiency of Algorithm~\ref{alg:inms} for solving the large-scaled GAVE~\eqref{eq:gave}. In \cite{wangcao2019}, it has shown the advantages of the MN iterative method over the GN iterative method \cite{mang2009}, the modified generalized Newton iterative method \cite{li2016} and the Picard iterative method \cite{rohn2014}. In \cite{li2021}, the NMN iterative method is compared with the MN iterative method. In \cite{wu2021n}, numerical experiments have verified that some relaxation versions of the NMS iterative method are superior to the Picard iterative method and the MN iterative method under certain conditions. In addition, as stated in the previous sections, the NMS method contains the MN and the NMN iterative methods as its special cases. Hence, in this section, we focus our attention on comparing the performance of the NMS iterative method and the INMS iterative method. Concretely, the following six algorithms will be tested.

\begin{enumerate}
\item NJ: The NJ method, namely, Algorithm~\ref{alg:nms} with $M = D$ and $N = L+U$.

\item INJ: The inexact version of the NJ method, that is, Algorithm~\ref{alg:inms} with $M = D$ and $N = L+U$.

\item NGS: The NGS method, namely, Algorithm~\ref{alg:nms} with $M = D-L$ and $N = U$.

\item INGS: The inexact NGS method, that is, Algorithm~\ref{alg:inms} with $M = D-L$ and $N = U$.

\item NSOR: The NSOR method, namely, Algorithm~\ref{alg:nms} with $M = \frac{1}{\alpha}D-L$ and $N =(\frac{1}{\alpha}-1)D+U$.

\item INSOR: The inexact NSOR method, that is, Algorithm~\ref{alg:inms} with $M = \frac{1}{\alpha}D-L$ and $N =(\frac{1}{\alpha}-1)D+U$.

\end{enumerate}

At each iterative step of NJ, NGS and NSOR methods, the main task is to exactly solve a system of linear equations with the coefficient matrix $D + \Omega$, $D + \Omega - L$ and $D + \alpha\Omega - \alpha L$, respectively. For the sake of efficiency, we can pre-compute the LU decomposition of the above-mentioned coefficient matrices using the \textbf{decomposition} function of MATLAB. For the inexact methods, theoretically, it follows from \eqref{eq:15} that
\begin{equation*}
	0\leq \theta < \frac{1-\|(\Omega+M)^{-1}\|(\|\Omega+N\|+\|B\|)}{\|(\Omega+M)^{-1}\| (\|\Omega+M\|+\|\Omega+N\|+\|B\|)}<1.
	\end{equation*}
However, $ \frac{1-\|(\Omega+M)^{-1}\|(\|\Omega+N\|+\|B\|)}{\|(\Omega+M)^{-1}\| (\|\Omega+M\|+\|\Omega+N\|+\|B\|)}$ is generally expensive to compute or hard to estimate. In practice, based on this theoretical guidance, $\theta = \min\left\{0.5, \frac{1}{\max\{1,k-l_{\max}\}}\right\}$ with $l_{\max} = 10$ is used. Here, $k$ counts the number of outer iterative step. In addition,
the LSQR \cite{pasa1982} method is used as the inner iterative method to approximately solve the involved linear systems.

In the numerical results, we report the number of iteration steps (denoted by ``IT"), the elapsed CPU time in seconds (denoted as ``CPU") and the relative residual (denoted by ``RES").
For the sake of fairness, the reported CPU time is the mean value of ten tests for each method. RES is defined as
\begin{equation*}
\text{RES}(x^k):= \frac{\|Ax^k - B|x^k| - b \|}{\|b\|}.
\end{equation*}
All tests are started from the initial vector $x^0 = (1,0,1,0,\cdots,1,0,\cdots)^\top$ and terminated if
$
\text{RES}(x^k)\le 10^{-6}
$
or the prescribed maximal iteration number $k_{\max} = 500$ is exceeded. All experiments are implemented in MATLAB R2018b with a machine precision $2.22\times 10^{-16}$ on a PC Windows 10 operating system with an Intel i7-9700 CPU and 8GB RAM.

\begin{exam}[\!\!\cite{wangcao2019}]\label{Example4.1}
 Consider the LCP(M,q), where $M=\hat{M}+\mu I$ with
 $$\hat{M}=\textbf{tridiag}(-I,S,-I)\in \mathbb{R}^{n\times n} \quad and \quad S=\textbf{tridiag}(-1,4,-1)\in \mathbb{R}^{m\times m},$$
and $q = -Mz^*$ with $z^* = (1.2, 1.2, \cdots, 1.2)^\top \in \mathbb{R}^n$ being the unique solution of the LCP(M,q). In this case, the unique solution of the corresponding  GAVE~\eqref{eq:gave} (with $A = M + I$, $B = M-I$ and $x = \frac{1}{2}\left[(M-I)z +q\right]$) is $x^* = (-0.6, -0.6, \cdots, -0.6)^\top \in \mathbb{R}^{n}$.

As stated in \cite{wu2021n}, $\Omega = \hat{M}$ and $\Omega=1.5\hat{M}$ are used. In addition, two values of the parameter $\mu$ are used, i.e., $\mu = 4$ and $\mu = -1$. For the NSOR method, the experimentally optimal parameter $\alpha_{exp}$ is used, which makes the NSOR method require the smallest iterative step. The same parameter is used for the INSOR method. Numerical results for this example are reported in Tables~\ref{table1}-\ref{table4}, from which we can find that the inexact methods are superior to the corresponding exact methods in terms of CPU time.
\end{exam}

\setlength{\tabcolsep}{1.5pt}
\begin{table}[!h]\small
	\centering
	\caption{Numerical results for Example~\ref{Example4.1} with $\mu=4$ and $\Omega=\hat{M}$.} \label{table1}
	\begin{tabular}{c c c c c c c c }\hline
		\multirow{1}{*}{Method}  & $n$ &  $10000$ & $12100$ & $14400$ & $16900$ & $19600$ & $22500$ \\\hline
		\multirow{3}{*}{NJ} & IT  & $12$ & $12$ & $12$ & $12$& $12$& $12$  \\
		& CPU & $0.0387$ & $0.0533$ & $0.0627$ & $0.0851$& $0.1140$& $0.1382$\\
		& RES & $6.7322\times 10^{-7}$ & $6.4359\times 10^{-7}$  & $6.1760\times 10^{-7}$ &  $5.9457\times 10^{-7}$& $5.7399\times 10^{-7}$ & $5.5545\times 10^{-7}$\\\hline
\multirow{3}{*}{INJ} & IT  & $23$ & $23$ & $23$ & $23$& $23$ & $23$\\
		& CPU & $\textbf{0.0136}$ & $\textbf{0.0167}$ & $\textbf{0.0175}$ & $\textbf{0.0202}$& $\textbf{0.0238}$  & $\textbf{0.0296}$\\
		& RES &$6.1803\times 10^{-7}$ &$6.0229\times 10^{-7}$ & $5.8752\times 10^{-7}$ & $5.7367\times 10^{-7} $& $5.6067\times 10^{-7}$ & $5.4846\times 10^{-7}$\\\hline
		\multirow{3}{*}{NGS} & IT  & $11$ & $11$ & $11$ & $11$& $11$& $11$\\
		& CPU & $0.0382$ & $0.0471$ & $0.0543$ & $0.0757$& $0.0970$ & $0.1079$ \\
		& RES &$3.3279\times 10^{-7}$ & $3.2923\times 10^{-7}$ & $3.2620\times 10^{-7}$ & $3.2361\times 10^{-7}$& $3.2135\times 10^{-7}$ & $3.1937\times 10^{-7}$\\\hline
\multirow{3}{*}{INGS} & IT  & $16$ & $16$ & $16$ & $16$& $16$ & $16$\\
		& CPU & $\textbf{0.0112}$ & $\textbf{0.0142}$ & $\textbf{0.0126}$ & $\textbf{0.0152}$& $\textbf{0.0186}$ & $\textbf{0.0202}$\\
		& RES &$4.7650 \times 10^{-7}$ &$4.5093\times 10^{-7}$ & $4.3174\times 10^{-7}$ & $4.1712\times 10^{-7}$& $4.0547\times 10^{-7}$ & $3.9561\times 10^{-7}$  \\ \hline
		\multirow{4}{*}{NSOR} &$\alpha_{exp}$ & $0.9$ & $0.9$ & $0.9$ & $0.9$ & $0.9$ & $0.9$\\
& IT  & $9$ & $9$ & $9$ & $9$& $9$ & $9$\\

		& CPU & $0.0361$ & $0.0453$ & $0.0529$ & $0.0706$& $0.0876$ & $0.1011$\\
		& RES &$1.8257\times 10^{-7}$ &$1.8105\times 10^{-7}$ & $1.7976\times 10^{-7}$ & $1.7865\times 10^{-7}$& $1.7769\times 10^{-7}$  & $1.7685\times 10^{-7}$\\\hline
		\multirow{3}{*}{INSOR} & IT  & $16$ & $16$ & $16$ & $16$& $16$ & $16$\\
		& CPU & $\textbf{0.0112}$ & $\textbf{0.0127}$ & $\textbf{0.0133}$ & $\textbf{0.0176}$& $\textbf{0.0210}$ & $\textbf{0.0217}$\\
		& RES &$4.7363\times 10^{-7}$ &$4.5537\times 10^{-7}$ & $4.3715\times 10^{-7}$ & $4.1948\times 10^{-7}$& $4.0268\times 10^{-7}$ & $3.8689\times 10^{-7}$ \\ \hline
	\end{tabular}
\end{table}

\setlength{\tabcolsep}{1.5pt}
\begin{table}[!h]\small
	\centering
	\caption{Numerical results for Example~\ref{Example4.1} with $\mu=4$ and $\Omega=1.5\hat{M}$.} \label{table2}
	\begin{tabular}{c c c c c c c c }\hline
		\multirow{1}{*}{Method}  & $n$ &  $10000$ & $12100$ & $14400$ & $16900$ & $19600$ & $22500$ \\\hline
		\multirow{3}{*}{NJ} & IT  & $8$ & $8$ & $8$ & $8$& $8$& $8$  \\
		& CPU & $0.0354$ & $0.0440$ & $0.0566$ & $0.0764$& $0.0939$& $0.1053$\\
		& RES & $4.3499\times 10^{-7}$ & $4.3739\times 10^{-7}$  & $4.3940\times 10^{-7}$ &  $4.4110\times 10^{-7}$& $4.4256\times 10^{-7}$ & $4.4383\times 10^{-7}$\\\hline
\multirow{3}{*}{INJ} & IT  & $14$ & $15$ & $15$ & $15$& $15$ & $15$\\
		& CPU & $\textbf{0.0108}$ & $\textbf{0.0126}$ & $\textbf{0.0149}$ & $\textbf{0.0175}$& $\textbf{0.0194}$  & $\textbf{0.0215}$\\
		& RES &$6.6261\times 10^{-7}$ &$1.5451\times 10^{-7}$ & $1.5122\times 10^{-7}$ & $1.5493\times 10^{-7} $& $5.7194\times 10^{-7}$ & $5.2089\times 10^{-7}$\\\hline
		\multirow{3}{*}{NGS} & IT  & $8$ & $8$ & $7$ & $7$& $7$& $7$\\
		& CPU & $0.0361$ & $0.0426$ & $0.0461$ & $0.0668$& $0.0817$ & $0.0870$ \\
		& RES &$1.5011\times 10^{-7}$ & $1.4560\times 10^{-7}$ & $9.9425\times 10^{-7}$ & $9.7057\times 10^{-7}$& $9.4977\times 10^{-7}$ & $9.3134\times 10^{-7}$\\\hline
\multirow{3}{*}{INGS} & IT  & $15$ & $15$ & $15$ & $16$& $16$ & $16$\\
		& CPU & $\textbf{0.0125}$ & $\textbf{0.0133}$ & $\textbf{0.0140}$ & $\textbf{0.0178}$& $\textbf{0.0207}$ & $\textbf{0.0223}$\\
		& RES &$4.2204 \times 10^{-7}$ &$4.4074\times 10^{-7}$ & $4.4622\times 10^{-7}$ & $5.3650\times 10^{-7}$& $3.3982\times 10^{-7}$ & $1.7201\times 10^{-7}$  \\ \hline
		\multirow{4}{*}{NSOR} &$\alpha_{exp}$ & $0.9$ & $0.9$ & $0.9$ & $0.9$ & $0.9$ & $0.9$\\
& IT  & $6$ & $6$ & $6$ & $6$& $6$ & $6$\\
		& CPU & $0.0335$ & $0.0428$ & $0.0450$ & $0.0618$& $0.0787$ & $0.0850$\\
		& RES &$4.8032\times 10^{-7}$ &$4.5455\times 10^{-7}$ & $4.3254\times 10^{-7}$ & $4.1323\times 10^{-7}$& $3.9632\times 10^{-7}$  & $3.8130\times 10^{-7}$\\\hline
		\multirow{3}{*}{INSOR} & IT  & $15$ & $15$ & $15$ & $15$& $15$ & $15$\\
		& CPU & $\textbf{0.0127}$ & $\textbf{0.0139}$ & $\textbf{0.0146}$ & $\textbf{0.0171}$& $\textbf{0.0201}$ & $\textbf{0.0223}$\\
		& RES &$7.9187\times 10^{-7}$ &$7.5056\times 10^{-7}$ & $7.1291\times 10^{-7}$ & $6.7883\times 10^{-7}$& $6.4953\times 10^{-7}$ & $6.2435\times 10^{-7}$ \\ \hline
	\end{tabular}
\end{table}

\setlength{\tabcolsep}{1.5pt}
\begin{table}[!h]\small
	\centering
	\caption{Numerical results for Example~\ref{Example4.1} with $\mu=-1$ and $\Omega=\hat{M}$.} \label{table3}
	\begin{tabular}{c c c c c c c c }\hline
		\multirow{1}{*}{Method}  & $n$ &  $10000$ & $12100$ & $14400$ & $16900$ & $19600$ & $22500$ \\\hline
		\multirow{3}{*}{NJ} & IT  & $50$ & $50$ & $50$ & $50$& $50$& $49$  \\
		& CPU & $0.0935$ & $0.1198$ & $0.1611$ & $0.1867$& $0.2253$& $0.2694$\\
		& RES & $9.0284\times 10^{-7}$ & $8.7295\times 10^{-7}$  & $8.4716\times 10^{-7}$ &  $8.2466\times 10^{-7}$& $8.0483\times 10^{-7}$ & $9.9690\times 10^{-7}$\\\hline
\multirow{3}{*}{INJ} & IT  & $48$ & $48$ & $48$ & $48$& $48$ & $48$\\
		& CPU & $\textbf{0.0406}$ & $\textbf{0.0429}$ & $\textbf{0.0494}$ & $\textbf{0.0534}$& $\textbf{0.0608}$  & $\textbf{0.0666}$\\
		& RES &$8.2690\times 10^{-7}$ &$8.2886\times 10^{-7}$ & $8.3015\times 10^{-7}$ & $8.3646\times 10^{-7} $& $8.3260\times 10^{-7}$ & $8.3539\times 10^{-7}$\\\hline
		\multirow{3}{*}{NGS} & IT  & $57$ & $57$ & $57$ & $56$& $56$& $56$\\
		& CPU & $0.1098$ & $0.1290$ & $0.1613$ & $0.1993$& $0.2481$ & $0.3003$ \\
		& RES &$9.2895\times 10^{-7}$ & $8.8685\times 10^{-7}$ & $8.5001\times 10^{-7}$ & $9.6738\times 10^{-7}$& $9.3301\times 10^{-7}$ & $9.0209\times 10^{-7}$\\\hline
\multirow{3}{*}{INGS} & IT  & $61$ & $60$ & $60$ & $59$& $58$ & $58$\\
		& CPU & $\textbf{0.0632}$ & $\textbf{0.0640}$ & $\textbf{0.0715}$ & $\textbf{0.0763}$& $\textbf{0.0858}$ & $\textbf{0.0946}$\\
		& RES &$9.7209\times 10^{-7}$ &$9.1956\times 10^{-7}$ & $8.6591\times 10^{-7}$ & $9.6978\times 10^{-7}$& $9.1564\times 10^{-7}$ & $9.0777\times 10^{-7}$  \\ \hline
		\multirow{4}{*}{NSOR} &$\alpha_{exp}$ & $1.3$ & $1.29$ & $1.29$ & $1.29$ & $1.28$ & $1.24$\\
& IT  & $53$ & $52$ & $52$ & $52$& $52$ & $52$\\
		& CPU & $0.0974$ & $0.1221$ & $0.1564$ & $0.1935$& $0.2463$ & $0.3009$\\
		& RES &$8.5697\times 10^{-7}$ &$9.9468\times 10^{-7}$ & $9.6795\times 10^{-7}$ & $9.9154\times 10^{-7}$& $9.0375\times 10^{-7}$  & $9.4811\times 10^{-7}$\\\hline
		\multirow{3}{*}{INSOR} & IT  & $57$ & $53$ & $56$ & $56$& $56$ & $51$\\
		& CPU & $\textbf{0.0636}$ & $\textbf{0.0624}$ & $\textbf{0.0735}$ & $\textbf{0.0824}$& $\textbf{0.0907}$ & $\textbf{0.0874}$\\
		& RES &$9.6355\times 10^{-7}$ &$9.4753\times 10^{-7}$ & $8.4190\times 10^{-7}$ & $8.4610\times 10^{-7}$& $8.3985\times 10^{-7}$ & $8.6240\times 10^{-7}$ \\ \hline
	\end{tabular}
\end{table}

\setlength{\tabcolsep}{1.5pt}
\begin{table}[!h]\small
	\centering
	\caption{Numerical results for Example~\ref{Example4.1} with $\mu=-1$ and $\Omega=1.5\hat{M}$.} \label{table4}
	\begin{tabular}{c c c c c c c c }\hline
		\multirow{1}{*}{Method}  & $n$ &  $10000$ & $12100$ & $14400$ & $16900$ & $19600$ & $22500$ \\\hline
		\multirow{3}{*}{NJ} & IT  & $67$ & $66$ & $66$ & $66$& $66$& $65$  \\
		& CPU & $0.1204$ & $0.1607$ & $0.2001$ & $0.2622$& $0.2942$& $0.3530$\\
		& RES & $8.5918\times 10^{-7}$ & $9.6445\times 10^{-7}$  & $9.3898\times 10^{-7}$ &  $8.9779\times 10^{-7}$& $8.7010\times 10^{-7}$ & $9.9392\times 10^{-7}$\\\hline
\multirow{3}{*}{INJ} & IT  & $68$ & $68$ & $68$ & $68$& $69$ & $66$\\
		& CPU & $\textbf{0.0689}$ & $\textbf{0.0713}$ & $\textbf{0.0796}$ & $\textbf{0.0974}$& $\textbf{0.1007}$  & $\textbf{0.1056}$\\
		& RES &$8.7035\times 10^{-7}$ &$9.8651\times 10^{-7}$ & $9.4970\times 10^{-7}$ & $9.3371\times 10^{-7} $& $9.7098\times 10^{-7}$ & $9.0118\times 10^{-7}$\\\hline
		\multirow{3}{*}{NGS} & IT  & $74$ & $74$ & $73$ & $73$& $73$& $72$\\
		& CPU & $0.1454$ & $0.1751$ & $0.2207$ & $0.2994$& $0.3546$ & $0.4199$ \\
		& RES &$9.2533\times 10^{-7}$ & $8.8342\times 10^{-7}$ & $9.6383\times 10^{-7}$ & $9.2693\times 10^{-7}$& $8.9398\times 10^{-7}$ & $9.8427\times 10^{-7}$\\\hline
\multirow{3}{*}{INGS} & IT  & $79$ & $81$ & $82$ & $82$& $82$ & $82$\\
		& CPU & $\textbf{0.0950}$ & $\textbf{0.0957}$ & $\textbf{0.1068}$ & $\textbf{0.1254}$& $\textbf{0.1352}$ & $\textbf{0.1466}$\\
		& RES &$9.7577 \times 10^{-7}$ &$8.8595\times 10^{-7}$ & $8.8570\times 10^{-7}$ & $9.1197\times 10^{-7}$& $8.9751\times 10^{-7}$ & $8.9341\times 10^{-7}$  \\ \hline
		\multirow{4}{*}{NSOR} &$\alpha_{exp}$ & $1.3$ & $1.3$ & $1.3$ & $1.3$ & $1.3$ & $1.3$\\
& IT  & $69$ & $69$ & $69$ & $69$& $68$ & $68$\\
		& CPU & $0.1237$ & $0.1583$ & $0.1957$ & $0.2491$& $0.3043$ & $0.3525$\\
		& RES &$9.9378\times 10^{-7}$ &$9.4844\times 10^{-7}$ & $9.0873\times 10^{-7}$ & $8.7362\times 10^{-7}$& $9.6738\times 10^{-7}$  & $9.4340\times 10^{-7}$\\\hline
		\multirow{3}{*}{INSOR} & IT  & $61$ & $63$ & $63$ & $67$& $67$ & $66$\\
		& CPU & $\textbf{0.0747}$ & $\textbf{0.0787}$ & $\textbf{0.0887}$ & $\textbf{0.1078}$& $\textbf{0.1157}$ & $\textbf{0.1292}$\\
		& RES &$9.6623\times 10^{-7}$ &$9.6313\times 10^{-7}$ & $9.8202\times 10^{-7}$ & $8.8519\times 10^{-7}$& $8.7996\times 10^{-7}$ & $9.6289\times 10^{-7}$ \\ \hline
	\end{tabular}
\end{table}

\section{Conclusions}\label{sec:conclusion}
An inexact framework of the Newton-based matrix splitting (INMS) iterative method is developed for solving the GAVE~\eqref{eq:gave}. The INMS iterative method can be regarded as a generalization of the exact NMS iterative method proposed in \cite{wangcao2019}. Linear convergence of the INMS iterative method is studied in detail. Numerical results show that the INMS method is superior to the exact NMS method in terms of CPU time.



\clearpage

\end{document}